\documentclass[12pt]{amsart}

\usepackage{amsmath,amsthm,amssymb}
\usepackage[latin9]{inputenc} 
\usepackage{verbatim}
\usepackage{pst-all}  
\theoremstyle{plain}
\newtheorem{lemma}{Lemma}[section]

\newtheorem{corollary}[lemma]{Corollary}

\newtheorem{theorem}[lemma]{Theorem}
\theoremstyle{definition}
\newtheorem{definition}{Definition}
\theoremstyle{remark}
\newtheorem{remark}[lemma]{Remark}

\newcommand{\calA}{\mathcal{A}}

\newcommand{\calC}{\mathcal{C}}

\newcommand{\N}{\mathbb{N}}
\newcommand{\mathN}{\mathbb{N}}

\newcommand{\R}{\mathbb{R}}

\newcommand{\mathC}{\mathbb{C}}

\newcommand{\calK}{\mathcal{K}}

\newcommand{\FU}{FU}

\newcommand{\XU}{XU}
\def \Poisstr {{\bs K}}
\newcommand{\bs}{\boldsymbol}

\DeclareMathOperator{\dist}{\mathrm {dist}}
\DeclareMathOperator{\distnu}{\mathrm {d}_\nu}

\DeclareMathOperator{\Interior}{\mathrm {Int}}

\numberwithin{equation}{section}

\begin{document}
 \baselineskip=17pt
\fontsize{11}{12}
\selectfont

 \setcounter{page}{1}

\title[Universal harmonic functions and martingales on trees]{Frequently dense harmonic functions and universal martingales on trees}
\author{Evgeny Abakumov, Vassili Nestoridis, Massimo~A. Picardello}
\address{LAMA, Universit\'{e} Gustave Eiffel, Universit\'{e} Paris Est Creteil, CNRS\\
F-77454, Marne-la-Vall\'{e}e\\France}
\email{evgueni.abakoumov@u-pem.fr}
\address{Mathematics Department\\
National and Kapodistrian University of Athens, Panepistimioupolis, \\GR-15784, Athens\\Greece} \email{vnestor@math.uoa.gr}
\address{Dipartimento di Matematica\\
Universit\`a di Roma ``Tor Vergata'', Via della Ricerca
Scientifica\\00133~Roma\\Italy} \email{picard@mat.uniroma2.it}
\subjclass{05C05, 31A20, 60J45}
\keywords{Non-homogeneous trees, transient transition operators, harmonic functions,
boundary of a tree, Poisson transform, martingale associated with a harmonic function,
universal functions on trees}

\thanks{ 
The first author  is partially supported by 
project ANR-18-CE40-0035.
The last author is partially supported by MIUR Excellence Departments Project awarded to
the Department of Mathematics, University of Rome Tor Vergata, CUP E83C18000100006
}

\begin{abstract}
On a large class of infinite trees $T$, we prove the existence of harmonic functions $h$, with respect to suitable transient transition operators $P$,
that satisfy 
the following universal property:  $h$ is  the Poisson transform of a martingale on the end-point boundary $\Omega$ of $T$ (equipped with the harmonic measure induced by $P$) t such that, for every measurable function $f$ on $\Omega$, it contains a subsequence converging to $f$ in measure.  Moreover, the martingale visits every open set of measurable functions with positive lower density.
\end{abstract}

\maketitle

\section{Introduction.}\label{Sect:Intro}
We start by outlining our main ideas and goals: precise definitions will be given later. This paper studies density, frequent density and universal properties of martingales and harmonic functions with respect to nearest neighbor transition operators on an infinite tree $T$. The simplest tree is a chain of vertices, isomorphic to the integers. Consider, for simplicity, the symmetric nearest neighbor transition operator with equal probability to move to each of the two neighbors (the \emph{isotropic} operator): then its harmonic functions are linear functions, whose image cannot be dense in $\mathC$. Similarly, density fails for any nearest neighbor transition operator. But if a tree has infinitely many bifurcation vertices (for instance, a binary tree), and we start with any arbitrary value of a function $f$ at a given vertex, then it is clear that there are enough degrees of freedom to extend $f$ harmonically at each vertex. More generally, let us consider a nearest neighbor transition operator with strictly positive transition coefficients from every vertex to each neighbor,
it was shown in \cite{ANP} that there is a large class of infinite subsets $R$ of $T$ such that every function with preassigned values therein can be extended to a harmonic function on $T$; it is trivial to choose the values so that the image is dense. Moreover, given any enumeration of these values, the image can be chosen to be frequently dense in that enumeration. 

More precisely, the sets  $R=\{r_1,r_2,\dots\}\subset T$ introduced in \cite{ANP}, called \emph{ramified}, are those with enough bifurcations, in the sense that they are not contained in the union of a finite set and finitely many \emph{linear branches} (that is, chains of contiguous vertices  each of which has only two neighbors). It was proved in \cite{ANP} that
 the set 
$D(R)=D(R, T, P)$ of $P-$harmonic functions $f$ on $T$ (defined in Subsection \ref{SubS:harmonic})
with $f(R)$ dense in $\mathC$  is non-empty if and only if $R$ is ramified. A function $f\in D(R)$ is called \emph{$R$-frequently dense} for a given enumeration of $R$ if, for each non-empty open $V\subset\mathC$, the set of integers $B(f,V)=\{j\in\mathN\colon  f(r_j)\in V\}$ has \emph{strictly positive lower density}  in this enumeration; that is, $\liminf_{n\to +\infty}  
 |\{j\in B(f,V) \colon j\leqslant n\}|  
 /n>0$.
It was also proved in \cite{ANP} that, if  $R$ is ramified,  the following density properties hold: 
$D(R)$ is  a dense $G_\delta$ subset of  the space of $P-$harmonic functions $H_P(T)$,  $D(R)\cup \{0\}$ contains a dense vector subspace of $H_P(T)$, and if $R$ ramifies frequently enough (namely, 
if $R$ is an infinite subset of $T$ such that every linear branch of $T$
contains at most $C$ elements of $R$ for some $C>0$ independent of the linear branch), then for any enumeration of $R$ the set
$FD(R)$ of its frequently dense harmonic functions 
is dense in $H_P(T)$
(otherwise, there always exists an enumeration  of $R$
for which $FD(R)$ is empty).

All this was proved in \cite{ANP} for all nearest neighbor transition operators such that every vertex has strictly positive transition probability to each neighbor. But if, in addition,  the transition operator $P$ is \emph{transient}, that is if it generates a random walk going to the boundary at infinity $\Omega=\Omega(T)$ (that consists of all geodesic rays starting at any fixed reference vertex), then $P-$harmonic functions can be reconstructed from their boundary values; more precisely, they are obtained by finitely additive measures on $\Omega$, in the sense that they are  Poisson integrals of finitely additive boundary measures, also called distributions. 
Conversely, every distribution on the boundary gives rise to a $P-$harmonic function via the Poisson integral.

Here is another way to regard boundary distributions of harmonic functions on a tree, in terms of \emph{martingales} on its boundary $\Omega$. Indeed, the choice of a reference vertex $o$ makes it meaningful to consider its descendants $v$ of any given generation $n\geqslant 0$ (for $n=0$ one has $v=o$); we say that these vertices have length $n$. Then $\Omega$ splits as the disjoint union of the subsets $\Omega(v)$ of all geodesic rays starting at $o$ and containing $v$, that we call \emph{arcs of the $n$-th generation}. This splitting induces on $\Omega$ a totally disconnected compact topology. Consider the family of nested $\sigma$-algebras $\calA_n$ generated by the arcs of the $n$-th generation. By integration on arcs of the $n$-th generation, every finitely additive measure $\nu$ on $\Omega$ projects to an $\calA_n$-measurable function $f_n$ that converge to $\nu$ in the sense of distributions, that is, in measure. Let us denote by $\pi_n$ these projections. Then, clearly, $\pi_n\pi_m\nu=\pi_m\pi_n\nu=\pi_{\min\{n,m\}}\nu$; in other words, the sequence of projections $\{\pi_n\nu\}$ is a martingale.

Then we can start with martingales on $\Omega$ instead of harmonic functions on $T$, and use the projections of a martingale, or equivalently a finitely additive measure $\phi$, to build a sequence of functions $\{f_n=\pi_n\phi\}$  on $\Omega$. Once a reference vertex $o$ is chosen and fixed, each such function
 can be naturally regarded as a function $f^\dagger_n$ on the vertices of length $n$ (and conversely, every function $h$ defined on the set $C_n$ of all vertices of length $n$ lifts to a function $h^*_n$ on $\Omega$ that is constant on arcs of the $n$-th generation).
The collection of all the $f^\dagger_n$, each defined on  $C_n$, gives rise to a unique function $f^\dagger$ on $T$, and it is immediate to see that $f^\dagger$ is harmonic with respect to a suitable transition operator. 
Similarly, every (normalized) positive measure $\nu$ on $\Omega$ (hence a Borel measure) gives rise to a martingale whose projections yield a function $\nu^\dagger$ on the vertices of $T$ and a forward-only nearest neighbor transition operator $Q$, necessarily transient; $\nu^\dagger$ is a $Q-$harmonic function. Note that    $\nu$ is the hitting distribution on $\Omega$ of the random walk induced by $Q$ starting at $o$.  Instead, if $\nu$ is not positive, we still obtain an operator $Q$ and a $Q-$harmonic function, but the transition coefficients of $Q$ are no longer positive.

To summarize: once a vertex $o$ is chosen, every positive Borel measure on the boundary projects to a function on $T$ that is $Q$-harmonic for a suitable (unique) forward-only transition operator; moreover, by the same family of projections (based on the measure $\nu$ on the boundary given by the hitting distribution of $Q$), every measurable function on $\Omega$ becomes a $Q-$harmonic function on $T$, and the same representation holds for measures or distributions (i.e., finitely additive measures) on $\Omega$.
 This function depends on the choice of $o$, but the corresponding harmonic function under a different choice of $o$ is the Poisson transform of an equivalent 
 measure or distribution on $\Omega$.

We  shall prove universal properties and frequent universal properties for  martingales on the boundary of trees with finite linear branches. As a consequence of the identification of harmonic functions and martingales, these universal properties can be simply rephrased in terms of harmonic functions of forward-only transition operators. We shall systematically identify boundary martingales with $Q-$harmonic functions ($Q$ being forward-only), and express all results in terms of these functions. Given any sequence $\bs n$ of radii, we consider the set $U(\bs n)$ of functions $h \in H_Q$ whose restrictions to the circles $C_{n_j}$, $n_j\in\bs n$,  regarded as locally constant functions on $\Omega$, are dense in the space of measurable functions on $\Omega$, and prove that this set is dense in $H_Q$ and that $U(\bs n) \cup \{0\}$ contains a dense vector subspace of $H_Q$. Moreover, the set of functions $h\in H_Q$ that visit every open subset of the space of measurable functions on $\Omega$ with positive lower density is also dense in $H_Q$. These results were announced in \cite{ANP}; they  can be regarded as discrete analogues of the results of \cite{B} on holomorphic functions.

Finally, we observe that the previous statements about frequently dense martingales reduce attention to the associated $Q-$harmonic functions. For most of the results, 
$Q$ is a forward-only operator with respect to a given root vertex. On the other hand, in Section \ref{Sec:Universality_for_very_regular_operator} we  extend the scope of our result to strictly positive nearest neighbor transition coefficients as follows. We have shown that, once we start with a positive Borel measure $\nu$ on $\Omega$, there is a unique forward-only transition operator $Q$ such that $\nu$ is the hitting distribution of the random walk generated by $Q$ and starting at $o$. But this is not the only transient operator that has $\nu$ as hitting distribution on the boundary; there is a large amount of such transient nearest neighbor transition operators $P$, and it was shown in \cite{Wo} (see also \cite[Section 7.2]{BGPW}) how to construct all of them explicitly. For each such $P$, the associated Poisson representation identifies $\nu$ with a $P-$harmonic function on $T$, and maps the space of all distributions on $\Omega$ onto the space $H_P$. So we can apply our methods for frequent universality to all $P-$harmonic functions for a huge class of transition operators $P$.

For this purpose, one of  the most natural classes of  operators is  the
\emph{very regular} nearest neighbor transition operators (that is, those whose transition probabilities to the neighbors are bounded from below by some $\delta>0$ and the backward transition probabilities are bounded (from above) above by $\frac12 - \delta$). These operators are transient, and are a subclass of those considered in \cite{ANP}. In Section \ref{Sec:Universality_for_very_regular_operator} we extend all our universality results to the spaces of  harmonic functions of very regular nearest neighbor transition operators, through estimates for their hitting distribution in the boundary \cite{KPT}.

\section{Notation and preliminaries}\label{Sect:preliminaries}
\subsection{Trees.}\label{SubS: Trees}

Notation on trees is not completely uniform; we follow most of the terminology established in \cite{AP, Ca, FP}. Here is
a review. A \emph{tree} $T$ is a connected, simply connected 
locally finite 
countable graph with no non-trivial loop and no terminal vertex (that is, each vertex has at least two neighbors).
With abuse of notation we shall also write $T$
for the set of vertices of the tree. In contrast with \cite{AP}, here we do not assume that $T$ is homogeneous; the number of edges
joining at every vertex of $T$ may vary, but stays finite. For $x$, $y \in T$ we write $x \sim y$ if $x$, $y$ are
neighbors. For any $x$, $y \in T$ there exist a unique $n \in \N$
and a unique minimal finite sequence $(z_0, \dots,z_n)$ of distinct
vertices such that $z_0=x$, $z_n=y$ and $z_k \sim z_{k+1}$ for all
$k<n$; this sequence is called the \emph{geodesic} path from $x$ to
$y$  and is denoted by $[x,y]$. The integer $n$ is called the
\emph{length} of $[x,y]$ and is denoted by $\dist(x,y)$; $\dist$ is a metric
on $T$. We fix a reference vertex $o \in T$ and call it the
\emph{origin}.
\label{origin}
The choice of $o$ induces a partial
ordering in $T$; $x \leqslant y $ if $x$ belongs to the geodesic
from $o$ to $y$. We denote by $y_-$ the unique neighbor of $y\neq o$ that belongs to this geodesic arc, that is, such that $y_-<y$ (the \emph{father} of $y$).

For $x \in T$, the length $|x|$ of $x$ is defined
as $|x|=d(o,x)$.  

\noindent The \emph{sector} $S(x)$ \label{sector}generated by a
vertex $x\neq o$ is the set of vertices $v$ such that $x\in [o,v]$.

For $k \in \N$ let $C_k$ be the
\emph{circle} $\{ x\in T\colon  |x|=k \}$, and $B_k$ the \emph{ball} $\{
x\in T\colon  |x|\leqslant k \}$.

We assign a stochastic nearest neighbour transition operator $P$, that is, a set of non-negative transition probabilities
$p(x,y)$ on the vertices $x,y\in T$  such that $p(x,y)=0$ if $x\not\sim y$ and $\sum_{y\colon  y\sim x} p(x,y)=1$ for every $x\in T$. This operator acts on functions $f$ on $T$ by the rule $P f(x) \equiv
\sum_{y\sim x} p(x,y)\,f(y)$.

\subsection{Harmonic functions}\label{SubS:harmonic}

\begin{definition}\label{Def:harmonic}  A function  $f\colon  T \to \R$ is harmonic at $x\in T$ if $P f = f$. A function is harmonic on $T$
if it is harmonic at every vertex of $T$, that is, if it is an eigenfunction of $P$ with eigenvalue 1.
\end{definition}

The space of harmonic functions on $T$ with respect to a transition operator $P$ is denoted by $H_P$. This space is equipped with the relative topology induced by the product (cartesian) topology of $\mathC^T$, which is a metric topology; the distance between $f$ and $g\in \mathC^T$ is given by
\begin{equation}\label{eq:metric}
\dist(f,g)=\sum_{j\in\mathN}\frac 1 {2^j}\,\frac{|f(x_j)-g(x_j)|}{1+|f(x_j)-g(x_j)|}\,,
\end{equation}
where $\{x_j\}$ is an enumeration of all vertices of $T$. The topology induced by this distance is the Cartesian topology and does not depend on the enumeration. Clearly, $H_P$ is closed in $\mathC^T$ and therefore is a metric space.

Note also that, since $T$ is countable, the space of all functions on $T$ (hence also the closed subspace $H_P$) is separable in this metric. A countable dense subset is the space of all finitely supported functions whose values have rational real and imaginary parts.

\subsection{Martingales on the boundary of a tree}\label{SubS:Martingales}
We define the boundary $\Omega$ of  tree $T$ as the set of infinite geodesics starting at $o$. In
analogy with the previous notation, for $\omega \in \Omega$ and
$n\in \N$, $\omega_n$ is the vertex of length $n$ in the geodesic
$\omega$. For $x \in T$  the \emph{interval} $I(x) \subset \Omega$,
generated by $x$, is the set $I(x)= \left\{ \omega \in \Omega\colon 
x=\omega_{|x|} \right\}$. The sets $I(\omega_n)$, $n \in \N$, form
an open base at $\omega$ of a topology of $\Omega$. Equipped with this topology
$\Omega$ is  compact and totally disconnected. 
The boundary arcs  $I(v)$ subtended by $v$ with $|v|=n$ form a $\sigma$-algebra $A_n$ (that is finite since the tree is locally finite), and these $\sigma$-algebras are nested; $A_n\subset A_{n+1}$ for every $n$. These nested $\sigma-$algebras generate the Borel $\sigma-$algebra $A$ on $\Omega$.

Choose any probability measure $\nu$ on $\Omega$, that is, a normalized positive Borel measure. Then 
the maps 
\begin{equation}\label{eq:projections}
\pi_n f= \sum_{|v|=n} \left( \frac 1{\nu(I(v))} \int_{I(v)} f \,d\nu\right)\chi_{I(v)}
\end{equation}
are   projections from
$L^1(\Omega)$ to locally constant functions in $L^\infty(\Omega)$ that are  $A_n$-measurable.

This family of projections forms a martingale: 
\begin{equation}\label{eq:martingale_property}
\pi_m\,\pi_nf=\pi_{\min\{m,n\}}f.
\end{equation}

For $f\in L^1(\Omega)$ we project the sequence $\pi_n f$ to a function on $T$ defined by
\begin{equation}\label{eq:projection:Omega->T}
f^\dag(v)=   \left. \vphantom{ \pi_{|v|}f|_{I(v)} } \pi_{|v|}f\right|_{I(v)}.
\end{equation}
Conversely, let
\begin{equation}\label{eq:forward_probabilities}
q(u,w)=\frac {\nu(I(w))}  {\nu(I(u))}
\end{equation}
for $u\sim w, u<w$. 
It is clear that $\sum_{w>v,\,w\sim v} q(u,w)=1$ and
\begin{equation}\label{eq:hitting_distribution_of_forward_only_operator_is_multiplicative}
\nu(I(v))=\prod_{j=1}^n q(v_{j-1},v_j)
\end{equation}
 if $v=v_n$ and $[v_0,v_1,\dots v_n]$
is its geodesic path from $o$. So the transition coefficients $q(v,w)$ are transition probabilities, and, if we set  $q(v,v_-)=0$, they form a
 transition operator $Q$ that is forward-only and nearest neighbour.

This defines a lifting to the boundary from functions $h$ on $T$ to sequences  $\{ h^*_{n}\colon  n\geqslant 0\}$ of $A_n$-measurable functions on $\Omega$, by the rule $ h_n:= h|_{C_n}$ and
\begin{equation}\label{eq:lifting}
 h_n^*(\omega):=h^*(\omega,\,n)=h(v) \qquad\textrm{if}\quad |v|=n,\quad\omega\in I(v).
 \end{equation}
 Note that here $\{h_n\}$ is a sequence of functions each defined on a different domain $C_n\subset T$. We use this way of writing in order to keep notation shorter; it will be used only for the harmonic function $h$ on $T$.   
The sequence $\{ h^*_n\}$ is a martingale if and only if $h$ is $Q$-harmonic; that is, forward harmonicity with respect to $Q$ is equivalent to the martingale property induced by $\nu$. 
This lifting is the inverse of the projection \eqref{eq:projection:Omega->T}, in the sense that $(h_{|v|}^*)^\dag(v)=h(v)$ for every $v\in T$, and more generally, by the martingale property \eqref{eq:martingale_property}, if $h\in H_Q$  
\begin{equation}\label{eq:martingale_property}
(h_{m}^*)^\dag(v)=h(v)  \quad\text{ for every $|v|\leqslant m$}. 
\end{equation}
Conversely, if $h$ is any function on $T$ and $h=(h_{m}^*)^\dag$ for some $m>0$, then $\{h^*_n\}$ satisfies the martingale property  \eqref{eq:martingale_property} for indices $n\leqslant m$, in other words $h$ is $Q$-harmonic in 
the ball $B_m$; this is actually the solution of the Dirichlet problem in $B_m$ for the operator $Q$.

\section{Universal harmonic functions in the sense of Menshov}\label{Sec:Menshov}

In this section we consider universal properties for martingales on the boundary of a tree $T$. We have seen in Subsection \ref{SubS:Martingales} that  
every probability measure $\nu$ on $\Omega$  gives rise to a forward-only nearest neighbor transition operator $Q$ 
and a
$Q-$harmonic function on $T$. 
Then $\nu$ is the hitting distribution on $\Omega$ of the random walk starting at $o$ generated by $Q$. In the same way, all other measures or distributions on $\Omega$ are in one-to-one correspondence (via the martingales that they generate) with the $Q-$harmonic functions. Therefore in this Section we  limit attention to the space $H_Q$ of harmonic functions of the generic forward-only nearest neighbor transition operator $Q$ on a rooted tree $T$ (with root vertex denoted by $o$). We shall assume that $q(v_-,\,v)\neq 0$ for every $v\neq o$; the results of \cite{ANP} extend to this set-up.
If there is an edge $e=[v_-,v]$ with $q(v_-,\,v) = 0$, then the problems that we shall discuss are equivalent to the same problems on the smaller sub-tree $T'$ obtained by \emph{pruning} $T$ at the edge $e$, that is, by discarding the vertex $v$ and all its descendants. Our universal results may hold for $T$ without holding for $T'$; 
;  for instance, it may happen that  the tree $T'$ is isomorphic  to the one-dimensional tree $\mathN$  (i.e., a one-sided infinite chain of vertices), so, as we will see, no universality results hold for $T'$.

The notion of universal property goes back to Fekete (before 1914) and Menshov (1945)
and has
became a large and popular area in analysis (see \cite{G} for more references). 
Menshov's paper \cite{Me45} proves
the existence of a trigonometric series whose partial sums approximate almost everywhere every $2\pi$-periodic function. 

It is well known that if a sequence $f_n$ defined on a space with measure $m$ converges to $f$ almost everywhere  then $f_n\to f$ in measure, that is $m\{x\colon |f_n-f|>\varepsilon\}$ tends to zero for every $\varepsilon>0$ as $n\to \infty$. On the other hand,  convergence in measure in a space of finite measure implies almost everywhere convergence for a subsequence. Therefore Menshov's universal property almost everywhere is equivalent to universal approximation in measure. In a similar sense (see also \cite{ANP}) we shall prove the existence of a universal $Q-$harmonic function $h$ on a tree, where $Q$ is a forward-only transition operator determined by the choice of $h$. The measure space is, of course, the boundary $\Omega$ of $T$, 
equipped with
 the probability measure $\nu$ given by the hitting distribution of $Q$  with starting vertex  $o$. Actually, we shall prove that, for every forward-only operator $Q$, the functions $h$ in a dense subset of $H_Q$, identified to a martingale with respect to the hitting distribution $\nu$ induced by $Q$ on $\Omega$, satisfy suitable universal properties in the space of $\nu-$measurable functions.

\begin{definition} \label {def:universal_functions} Let $M$ be the space of measurable functions on $\Omega$ and, as before, let $\nu$ be the hitting distribution on $\Omega$ of the random walk starting at $o$ generated by $Q$.
For any sequence $\bs{n}\subset\mathN$ let $U(\bs{n})=U(\bs{n},T)$ be the set of functions $h\in H_Q(T)$ such that  $\{h_n =h|_{C_n}\colon  n\in\bs{n}\}$ lifts to a sequence  $\{h_n ^*\}$ dense in $ M$; that is, for every $f\in M$ there is $h\in U(\bs{n})$ and  a subsequence $\{n_j\}$ of $\bs{n}$ such that $h_{n_j}$ converges to $f$ in measure (or equivalently, pointwise $\nu-$almost everywhere). 

We shall write $U$ in place of $U(\mathN)$. Elements of $U$  are called  \emph{universal functions};  elements of $U(\bs{n})$
are universal functions with respect to $\bs n$, or $\bs n$-universal functions.
\end{definition}

\begin{theorem}\label{theo:density_of_Utilde}
If all linear branches of $T$ have finite length, then for each sequence $\bs{n}\subset\mathN$  the set $U(\bs{n})$ is a dense $G_\delta$ subset of $H_Q$ (equipped, as usual, with the topology of pointwise convergence). If $T$ has an infinite linear branch, then $U$ is empty.
\end{theorem}

\begin{proof}
 The space $ M=( M,\nu)$ of measurable functions on $\Omega$, equipped with convergence in measure, is a complete metric space; its metric is 
\begin{equation}\label{eq:distance_of_measurable_functions}
\distnu(f,g)=\int_\Omega \frac{|f-g|}{1+|f-g|}\;d\nu.
\end{equation}
Note that, if $f=g$ in  $E\subset\Omega$, then
\begin{equation}\label{eq:triviality_on_distance_of_two_functions_that_coincide_in_a_subset}
\distnu(f,g) < \nu(\Omega\setminus E) = 1 - \nu(E).
\end{equation}
Let us denote by $\calK_n$ the space of $A_n-$measurable functions on $\Omega$.
Each measurable function on $\Omega$ can be approximated in measure by functions in some $\calK_n$. Therefore there is a sequence $f_j\in M$ dense in measure in $ M$ and such that, for every $j$, $f_j \in \calK_{n_j}$ for some $n_j$. We can always choose $n_j\in\bs n$ by choosing it larger if necessary, since $\calK_m\subset \calK_{m+1}$ for every $m$.

For every 
$j, s\in\mathN$ let $E(j,\,s)$ be the open balls
\[
E(j,\,s)=\{ f\colon \Omega\mapsto\mathC\text{ such that } \distnu(f,f_j)<1/s\}\,.
\]
It is clear that 
\[
U(\bs{n})= \cap_{s,j > 0} \cup_{n\in\bs{n}} \{h\in H_Q\colon  h_n ^* \in E(j,\,s)\}\,.
\]
Here, as before, we let $h_n $ be the restriction of $h$ to vertices of length $n$.
The lifting $h\to h_n ^*$ is continuous from $H_Q$ to $M$, hence $\{h\in H_Q\colon  h_n ^* \in E(j,\,s)\}$ is open. Therefore $U(\bs{n})$ is a $G_\delta$ in $H_Q$. By Baire's theorem, it is enough to show that $\cup_{n\in\bs{n}} \{h\in H_Q\colon  h_n ^* \in E(j,\,s)\}$ is dense in $H_Q$. From now on let us fix $j$ and $s$.

It is enough to show that, for every $s>0$, $N\in\mathN$ and $g\in H_Q$, there is $n\in\bs{n}$ and $h\in H_Q$ such that
$h_n ^*\in E(j,\,s)$ and $h=g$ in the ball $B_N=\{v\colon  |v|\leqslant N\}$.

Since $f_j\in \calK_{n_j}$, it is constant in the sectors subtended by vertices of length $n_j$. 
Choose $k > \log_2 s$  and $n\in\bs{n}$, $n\geqslant n_j$, $n\geqslant N+k$ and let $h=g$ in $B_{N-1}$. Let $|u|=N$. Since there are no infinite linear branches,
there is a descendant $u_1>u$ such that its father $(u_1)_-$ has at least two children and satisfies $q((u_1)_-,\,u_1) \leqslant 1/2$.
In the same way, there is a geodesic path containing $k$ vertices (not necessarily contiguous) $u=u_0< u_1 < u_2 < \dots < u_k(u)$ such that 
\begin{equation}\label{eq:decay_when_jumping_to_a_forward_neighbor}
q((u_i)_{-},\,u_{i})\leqslant \frac12
\end{equation}
for $1\leqslant i \leqslant k$. Let $|u_k(u)|=n(u)$ and let $m=\max \{n(u)\colon |u|=N\}$. Without loss of generality we can assume that all $u_k(u)$ for $|u|=N$ have length $m$ (otherwise replace $u_k(u)$ with one of its descendants of length $m$). By choosing a larger $m$ we can also assume $m\in\bs{n}$ and $m\geqslant n_j$, so $f_j\in\calK_m$.
Now, denote by $G_u$ the set of descendants of $u$ of length $|m|$. 
For $v\in G_u, \,v\neq u_k(u)$, let $h(v)=f_j^\dag(v)$. Observe that $h$ coincides with $f_j^\dagger$ on the vertices of length $m$ except $\{u_k(u)\colon |u|=N\}$, where it has not yet been defined. By the multiplicativity rule \eqref{eq:hitting_distribution_of_forward_only_operator_is_multiplicative} each arc subtended by $u_k(u)$ has measure less than $1/2^k$ times the measure of the arc subtended by $u$.
Therefore, whatever value we assign to $h$ on the exceptional vertices $u_k(u)$ for $|u|=N$, it follows from \eqref{eq:triviality_on_distance_of_two_functions_that_coincide_in_a_subset} that
 $\distnu(h_m ^*, \,f_j) \leqslant  1/2^k < 1/s$; if we can extend $h$ to a harmonic function, this shows that $h_m ^*\in E(j,\,s)$. Now, since $u':=u_k(u)$ is the only vertex in $G_u$ where $h$ is undefined, by iterating the forward-only harmonicity rule $k=m-N$ times we see that there exists one and only one value $h(u')$ such that
$q^{(k)}(u,u') h(u') + \sum_{v\in G_u,\,v\neq u'} q^{(k)}(u,v) h(v)  = g(u)$;
 indeed, here $q^{(k)}(u,v)$ is the $k$-th iterate $Q^k$ of the transition operator and is given by the product of  the coefficients $q(w,w_-)$ for $w$  in the path from $u$ to $u'$, $u$ excluded, hence it is non-zero.
 This value of
 $h(u_k(u))$ for $|u|=N$ makes the function $h:=(h_m ^*)^\dag$ satisfy the rule  $h(u)=(h_m ^*)^\dag(u)=\pi_{N} h_m ^*(u) = \pi_{N} g_m^*(u)=(g_m^*)^\dag(u)=g(u)$ by \eqref{eq:martingale_property}. In other words, the harmonic continuation of $h$ from $C_m$ to $B_m$ given by the solution of the Dirichlet problem (stated at the end of Section \ref{Sect:preliminaries}) satisfies $h=g$ in $C_N$. Then $h$ is harmonic in $B_{m-1}$, and since $h=g$ on $C_N$ and both functions are harmonic, $h=g$ in $B_N$ again by the maximum principle.  
 We only need to extend $h$ harmonically to all of $T$. Since $Q$ is forward-only, this is easy:
 we let $h$  be constant on each sector subtended by vertices $x$ of length $m$, by setting $h(y)=h(x)$ for every $y>x$. Hence $h$ is $Q$-harmonic everywhere, 
and $h_n ^*\in E(j,\,s)$ and $h=g$ in $B_N$.

The last sentence of the statement follows from the fact that harmonic functions of a forward-only transition operators are constant along linear branches.
\end{proof}

Now we shall prove algebraic genericity.
\begin{theorem}\label{theo:spaceability}
Under the hypothesis of Theorem \ref{theo:density_of_Utilde},  $U(\bs{n}) \cup \{0\}$  contains a dense vector subspace of $H_Q$.
\end{theorem}

\begin{proof} 
The space $H_Q$, equipped with the distance \eqref{eq:metric}, is separable. Choose a dense sequence $\{g_j\in H_Q\}$. Let $\bs{n}_0=\bs{n}$. By definition of $U(\bs{n})$, there is $h_1\in U(\bs{n})$ with $\dist(h_1,\,g_1)<1$, and a subsequence $\bs{n}_1\subset \bs{n}_0$ such that
$\lim_{n\in\bs{n}_1} (h_1)_{(n)}^* = 0$ (here the limit is in measure on $\Omega$). Then we build $h_2 \in 
U(\bs{n}_1)$ such that $\lim_{n\in\bs{n}_2} (h_2)_{(n)}^* = 0$ and $\dist (h_2,\,g_2)<1/2$, and, by iterating the argument, we produce nested sequences
$\bs{n}_0 \supset \bs{n}_1 \supset \dots $ and harmonic functions $h_j \in 
U(\bs{n}_j)\subset U(\bs{n})$ such that $\lim_{n\in\bs{n}_j} (h_j)_{(n)}^* = 0$ and $\dist(h_j,\,g_j)<1/j$.
Since the sequence $g_j$ is dense, the last inequality shows that also $\{h_j\}$ is dense in $H_Q$, hence so is its linear span $E$.

Now let $h\in H_Q$, $h\neq 0$, of type $h=\sum_{j=1}^m c_j h_j$, with, say, $c_m\neq 0$. We only need to show that $h\in U(\bs{n})$. Choose any measurable function $f$ on $\Omega$. Again by
definition of $U(\bs{n})$, there is a subsequence $\lambda_n\in \bs{n}_m \subset \bs{n}$ such that $(h_m )_{(\lambda_n)}^*\to f/c_m$ in measure. But we know that $(h_{j})_{(\lambda_n)}^*\to 0$ for $0\leqslant j <m$, by construction of the $h_j $'s and the fact that the sequences $\bs{n}_j$ are nested. It follows that $h_{\lambda_)}^*\to f$ in measure, hence $f\in U(\bs{n})$ and the proof is finished.
\end{proof}

We shall need the following elementary arithmetic lemma.
\begin{lemma}\label{lemma:lower_density_with_ell}
For $k\geqslant 1$, let $s\in\mathN$ be the power of 2 in the prime decomposition of $k$, that is,
the number such that $k/2^s$ is odd,  write $\ell(k)=s+1$, $r_1=1$, and set inductively \begin{equation}\label{eq:r_k}
r_k=r_{k-1}+\ell(k),
\end{equation}
that is, 
\begin{equation}\label{eq:r_k_as_sum}
r_k=\sum_{s=1}^k \ell(s).
\end{equation}
Then, for every $m\geqslant 1$, the set
$\{ r_n\colon  n\in\mathN,\; \ell(n)=m\}$
has strictly positive lower density in $\mathN$.
\end{lemma}
\begin{proof} 
We first observe that,
for $1\leqslant m\leqslant N$, 
\begin{equation}\label{eq:how_often_the_valies_repeat}
|\{ k \colon  1\leqslant k \leqslant 2^N, \;\ell(k)=m\}| = 2^{N-m}.
\end{equation}
Indeed, $\ell(k)=1$ only if $k$ is odd, hence exactly one half of the integers between 1 and $2^N$ satisfy $\ell(k)=1$ and so the identity above holds for $m=1$. Similarly, $\ell(k)= m$ if and only if $k$ is a multiple of $2^{m-1}$ but not of $2^{m-2}$, hence exactly a fraction $2^{-m}$ of the integers between 1 and $2^N$ satisfy $\ell(k)=m$, that is, $2^{N-m}$ integers.
Next, we show that, for every $N$,
\begin{equation}\label{eq:r_{2^N}}
r_{2^N} = 2^{N+1}-1\,.
\end{equation}
This is clear if $N=0$. We proceed by induction on $N$;  assuming $r_{2^N} = 2^{N+1}-1$, we show that $r_{2^{N+1}} = 2^{N+2}-1$.
Indeed, 
\begin{equation}\label{eq:r_{2^{N+1}}}
r_{2^{N+1}} = \sum_{k=1}^{2^{N+1}} \ell(k) = \sum_{k=1}^{2^{N}-1} \ell(k) + \ell(2^N) + \sum_{k=2^{N}+1}^{2^{N+1}-1} \ell(k) + \ell(2^{N+1}).
\end{equation}
Note that, if $k=2^s p$ for  $s<N$ and $p$ odd, then 
$k+2^N=2^s(p+2^{N-s})$, but $p+2^{N-2}$ is odd, hence $\ell(k+2^N)=s+1=\ell(s)$.
Therefore, if we write
$A=\sum_{k=1}^{2^{N}-1} \ell(k)$ and $B=\sum_{k=2^{N}+1}^{2^{N+1}-1} \ell(k)$, then $A=B$.
Moreover, by definition of $\ell$, we have $\ell(2^j)=j+1$ for every $j$. Therefore \eqref{eq:r_{2^{N+1}}} becomes
\begin{equation*} 
r_{2^{N+1}} = 2A + 2N + 3 = 1+ 2 \sum_{k=1}^{2^{N}} \ell(k) = 1 + 2r_{2^N},
\end{equation*}
by \eqref{eq:r_k_as_sum}. Now, by the induction hypothesis, $r_{2^{N+1}}=2^{N+2}-1$, that is \eqref{eq:r_{2^N}}.

Let $m\geqslant 1$.
Since $r_k$ is increasing, the statement follows if we prove that
\[
\liminf_{j\to\infty} \frac { |\{ n\in\mathN\colon   r_n \leqslant j,\; \ell(n)=m\}| }  {j} >0.
\]
Choose $j=r_{2^N}$. Then it is enough to show that
\[
\inf_{N\geqslant m} \frac { |\{ n\in\mathN\colon    r_n \leqslant r_{2^N},\; \ell(n)=m\}| }  {r_{2^N}} >0.
\]
But $r_{2^N}=2^{N+1}-1$, hence the last inequality holds if
\[
\inf_{N\geqslant m} \frac { |\{ n\in\mathN\colon    r_n \leqslant 2^N,\; \ell(n)=m\}| }  {2^N} >0.
\]
This inequality follows from \eqref{eq:how_often_the_valies_repeat}.
\end{proof}

We have seen that the martingale $\{h_n^*\}$ associated to a $\bs n$-universal function $h$ is dense in $M$, and, if linear branches have finite lengths, the space $U(\bs n)$ is dense in $H_Q$. Now we look at the subsequence of $\bs n$ of all indices $n$ such that $\{h_n^*\}$ approximates a given measurable function within a given tolerance. It will turn out that, if linear branches have bounded length, this subsequence has positive lower density.

\begin{definition}\label{def:frequently_dense}
Let $(M,\nu)$ be as in Definition \ref{def:universal_functions}.
A function $h\in H_Q$ is frequently universal harmonic 
if, for every non-empty open set $O\subset M$, the set $N(h,O)=\{n\in\mathN\colon  h_n^* \in O\}$ has positive lower density. The set of frequently universal harmonic functions is denoted by $\FU$.
\end{definition}

\begin{theorem}\label {theo:frequent_density}
If all vertices of $T$ have at least two descendants (that is, the root vertex has at least two neighbors and all other vertices at least three), then
 $\FU$ is dense in $H_Q$.
\end{theorem}
\begin{proof} 
We first show that $\FU$ is not empty. Let $h\in U$; in particular, the sequence $h_n ^*\in \calK_n$ is dense in $M$.
 For $g\in M$, denote by $B(g,\,s)$ the ball in $M$ with center $g$ and radius $s$. Then $f\in H_Q$ belongs to $\FU$ provided that
$f_{r_k}^* \in B(h^*_{\ell(k)},\;1/2^{\ell(k)})$, where the sequences $r_k$ and $\ell(k)$ are as in Lemma \ref{lemma:lower_density_with_ell}. The argument in the last part of the proof of Theorem \ref{theo:density_of_Utilde} shows that
every function $h$ defined in the ball $B_n\subset T$ and harmonic in $B_{n-1}$ can be extended, for every $m>0$ to a function, denoted again by $h$, defined in $B_{n+m}$ and harmonic in $B_{n+m-1}$, such that $h^*_{n+m} \in B(h^*_{n},\, 2^{-m})$. Now the existence of a frequently universal harmonic function $h$ follows by inductively applying this property to $n=r_{k-1}$, $m=\ell(k)$, and noticing that $n+m=r_k$ by \eqref{eq:r_k}, hence
there exists $h\in H_Q$ such that $h^*_{r_k} \in B(h^*_{\ell(k)},\, 2^{-\ell(k)})$ for every $k$.

We can repeat this construction by starting with a fixed harmonic function $g$ and setting $h^{(n)}=g$ in a ball $B_n$, and then extending $h^{(n)}$ to a frequently universal harmonic function. Since $h^{(n)}$ coincides with $g$ on a ball of radius $n$ and $n$ can be taken arbitrarily large, this shows that there exists a sequence $h^{(n)}\in \FU$ that converges to $g$, hence $\FU$ is dense in $H_Q$.
\end{proof}

\begin{remark} By a simple variant of the argument, Theorem \ref{theo:frequent_density} holds under the slightly more general assumption that there is a constant $c$ such that each linear branch in $T$ has length bounded by $c$. 
We just give an idea of why this is so. The general tree $\widetilde T$ with finite linear branches
is obtained from the tree $T$ without linear branches of Theorem \ref{theo:frequent_density} by inserting into each edge $e=[v,w]$ (where $w_-=v$) a chain of $m(v)-1$ consecutively adjacent vertices: in place of $e$ we have a linear branch of length $m(v)-1$, that is of $m(v)$ edges.  
The boundary $\widetilde\Omega$ of $\widetilde T$ is the same as the boundary $\Omega$ of $T$.
Let us now consider the simple case where $m(v)$ is radial around $o$, say $m(v)=m_n$ if $|v|=n$.
Here the martingales $\{h_n^*\}$ are sequences of functions on $\widetilde\Omega$ that repeat their values on the boundary arcs subtended by the vertices along each of the linear branches, that is $m_n$ times (actually, the arcs in each such chain coincide as sets, what changes is only their generation). Therefore this shifts the indices of the martingale.  If $m_n$ is bounded by some constant $c$, then the martingale repeats its values at most $c$ times at each generation, hence its variation is slower, but this does not affect frequent universality. But if $m_n$ is unbounded, then the martingale slows down at a higher and higher rate, and frequent universality may fail. 
 This trivial example shows that Theorem  \ref{theo:frequent_density} cannot hold without the assumption that
the linear branches of $T$ have bounded length.

More generally, if $m(v)$ is not necessarily radial but is bounded by $c$, we can prove frequent universality by replacing the sequence $h^*_{r_k}$ in the proof of Theorem \ref{theo:frequent_density} with $h^*_{N r_k}$, where $N$ is a suitable (large) integer depending on $c$.
\end{remark}

Let now 
$
\XU$ be the set of all $h\in H_Q$ such that  $N(h,O)$ has upper density 1 for every non-empty $O\subset  M.
$

\begin{theorem}\label{theo:meager_in_martingale_space}
Under the hypothesis of of Theorem \ref {theo:frequent_density}, 
$\XU$ is a dense $G_\delta$ subset of $H_Q$ disjoint from $\FU$, and so
 $\FU$ is meager  in $H_Q$.
\end{theorem}
\begin{proof}
Choose a countable family of non-empty open sets in $O_j\subset M$, and, in analogy with the proof of Theorem \ref{theo:density_of_Utilde}, let
\begin{align*}
E(j,m,n)=\{f\in H_Q(T)\colon  &\exists\, q>\bigl(1-\frac1m\bigr)\,n \text{ and } 1\leqslant k_1 \leqslant k_2\dots \leqslant k_q
 \leqslant n\colon \notag  \\[.1cm]
 & f^*_{(r_{k_l})}\in O_j \quad\text{for every }\,l=1,\dots,q\}.
\end{align*}
 Then each $E(j,m,n)$ is open in $H_Q(T)$ and 
 \begin{equation}\label{eq:XU_as_intersection_of_unions_of_sets_E(j,m,n)}
 \XU=\cap_{j,m,N}\cup_{n\geqslant N} E(j,m,n).
 \end{equation}
 Therefore $XU$ is a $G_\delta$ set.

 We claim that $\FU$ is disjoint from $\XU$ (see also \cite{KNP}). Indeed, let $O_1$ and $O_2$ be disjoint non-empty open sets in $\mathC$. Clearly, $N(f,O_1)\cap N(f,O_2)=\emptyset$. By frequent density, there exist $\delta>0$ and $m\in\mathN$ such that, for every $N\geqslant m$, 
\begin{equation*}\label{eq:frequent_density}
\frac{|\{n\leqslant N,\,n\in N(f,\,O_1)\}|}N > \frac{\delta}2\;.
\end{equation*}
Now assume that $f$ is also in $\XU$. This means that
\begin{equation*}\label{eq:upper_density_1}
\frac{|\{n\leqslant N_0,\,n\in N(f,\,O_2)\}|}{N_0} > 1- \frac{\delta}2
\end{equation*}
for some $N_0\geqslant m$. Choose $N=N_0$; since $N(f,\,O_1)$ and $N(f,\,O_2)$ are disjoint, this means that $|\{n\leqslant N_0,\,n\in N(f,\,O_1)\cup N(f,\,O_2) \}| > N_0$, a contradiction. This proves the claim.

Let us prove that $\XU$ is dense in $H_Q$ (see also \cite{Pa}).  By \eqref{eq:XU_as_intersection_of_unions_of_sets_E(j,m,n)} and Baire's theorem, it suffices to show that $\cup_{n\geqslant N} E(j,m,n)$ is dense in $H_Q$ for every $j,m,N$.

It is clear that
density is equivalent to the following property:  for every  $\varepsilon>0$, $j,m,N>0$, $g\in H_Q$ and for every finite sequence $v_1,\dots,v_N$ of vertices, there is $h\in H_Q$, and $n\geqslant N$ such that $|g(v_i)-h(v_i)|<\varepsilon$ for $i=1,\dots,N$ and $h\in E(j,m,n)$.
But this follows from precisely the same argument of Theorem \ref{theo:density_of_Utilde}.
Now the claim implies that $FU$ is meager, and the proof is complete.
\end{proof}

\section{Frequently universal harmonic functions of very regular transition operators}\label{Sec:Universality_for_very_regular_operator}

In the previous Section we have proved universal properties for harmonic functions of forward-only nearest neighbor transition operators on trees, naturally associated to boundary martingales. Here we briefly show how to extend these results to harmonic functions for a large class of transient nearest neighbor transition operators that are not forward-only.
We begin by recalling some known results on harmonic functions on trees equipped with transient transition operators and their boundary representation; most of the results are taken from \cite{Ca,KPT}.

We suppose that the transition operator $P$ on $T$ is transient, in the sense that the random walk that it generates leaves every finite set with positive probability. It is well known that the assumption of transience is equivalent to the existence of non-constant positive harmonic functions (see 
\cite{KSK}). If $P$ is transient and $v_n$ is the random vertex at time $n\in \mathN$ visited by the random walk starting at $v_0$, then the probability $U(v_0\,,\,v)$ of first hit at $v$ satisfies the inequality
\begin{equation}\label{eq:hitting_probability}
U(v_0\,,\,v):= \Pr\{\exists\, n\in\mathN\colon  v_n=v,\, v_j\neq v \;\forall\,j=1,\dots,n\}<1
\end{equation}
for every $v\in T$. We call a finite set of vertices $\calC$ a \emph{contour} if no two vertices of $\calC$ are neighbors and $T\setminus \calC$ splits as union of finitely many connected components only one of which (called the interior $\Interior\calC$ of $\calC$) is bounded. If $\calC$ is a contour and $v_0\in\Interior\calC$, then $v\mapsto U(v_0\,,\,v)$ is clearly a probability on $\calC$. If $h$ is harmonic in $\Interior\calC$, then $h(v_0)=\sum_{v\colon v\in\calC} U(v_0\,,\,v)\,h(v)$; this summation formula is the solution of the Dirichlet problem on $\Interior\calC$ \cite{PW}.

It is clear that, for each contour $\calC$, every harmonic function on $\calC\cup \Interior \calC$ attains its maximum on $\calC$.

As before, let $v_n$ be the random vertex at time $n$ of the random walk induced by $P$ starting at  $v_0$. Since $P$ is transient, there exists $v_\infty:=\lim_n v_n\in \Omega$ almost surely. The \emph{hitting distribution} $\nu_{v_0}$ is the probability measure on $\Omega$ (with respect to the Borel $\sigma-$algebra) defined on Borel sets $O\subset\Omega$ by $\nu_{v_0} (O) = \Pr\{v_\infty \in O\}$. The measures $\nu_v$ are mutually absolutely continuous, and the Radon--Nikodym derivative $K_{v_0}(v,\,\omega)=d\nu_v/d\nu_{v_0}(\omega)$ is called the \emph{Poisson kernel} induced by $P$.
If $h$ is a non-negative $P-$harmonic function, then the probabilistic Fatou theorem states that $\lim_m h_m^*(V_\infty)=\lim_n h(V_n)$ exists $\nu_{v_0}$ almost surely.

From now on we write $K(v,\,\omega)$ instead of $K_o(v,\,\omega)$, and $\nu$ instead of $\nu_o$. We set
$\Poisstr  h^* (v)= \int_\Omega h^*(\omega)\,K(v,\,\omega)\,d\nu_{v}(\omega)$; the operator $\Poisstr$, from $L^1(\Omega,\,\nu)$ to $H_P(T)$, is called the \emph{Poisson transform} induced by $P$.

A stopping time argument shows that $U(v,w)=U(v,u)\,U(u,w)$ whenever $u$ belongs to the geodesic path between $v$ and $w$. For every vertex $v\neq o$, consider the sector $S(v)=\{u\colon v\leqslant u\}$. Then $K(v,w):=U(v,w)/U(o,w)$ is constant on each sector $S(u)$ with $|u|=|v|$.
It follows from  the definition of $\nu$ that $K(v,\,\omega)=\lim_{w\to\omega} K(v,w)$. Hence the Poisson kernel is locally constant on $\Omega$; more precisely, it is constant on each arc $I(u)$ with $|u|=|v|$, that is, it is $A_{|v|}-$measurable.
 From this, the celebrated (deterministic) Fatou theorem follows (see \cite{Fa} for its analog on the disc): for every $h\in H_P(T)$ and $v\in T$,  one has $h(v)=\int_\Omega h^*(\omega)\,d\nu_{v}(\omega)= \int_\Omega h^*(\omega)\,K_{o}(v,\,\omega)\,d\nu_{o}$, where $h^*$ denotes the martingale associated to $h$ introduced in \eqref {eq:lifting}; here the integral of the martingale makes sense because $K_{o}(v,\,\omega)$ is locally constant on $\Omega$ for every $v$.

Then, for every harmonic function $h$ one has $\Poisstr h^* (v) = \Poisstr (\pi_{|v|} h^*) (v)$; in this way, the Poisson transform induced by $P$  of $h^*$ can be regarded as the Poisson transform of its associated martingale. It is clear that the $P$-Poisson integral of every martingale is $P-$harmonic (see also \cite{Ca}). Conversely, for every transient operator $P$, it was proved in \cite{KPT,PW} that every $P-$harmonic function (not only positive) is the $P$-Poisson transform of a martingale. 
But we have seen in Subsection \ref{SubS:Martingales} that every probability measure $\nu$ on $\Omega$ gives rise to a forward-only transition operator $Q$ whose harmonic functions are in bijective correspondence with the boundary martingales defined by $\nu$ via \eqref{eq:projections}. Note that different very regular operators $P$ can give rise to the same hitting distribution $\nu$ on $\Omega$, hence to the same forward-only transition operator $Q$. If $[v_0,v_1,\dots,v_n=v]$ denotes the geodesic path from $v_0$ to $v$, it follows from the fact that $Q$ is forward-only that 
\begin{equation}\label{eq:factoring_of_U}
U(v_0,v)=\prod_{j=1}^{n-1} q(v_{j-1},v_j)
\end{equation}
Hence, by \eqref{eq:hitting_distribution_of_forward_only_operator_is_multiplicative}, the measure $\nu$ is the hitting distribution of $Q$ on $\Omega$, and the family of projections in \eqref{eq:projections} gives rise to the Poisson transform induced by $Q$. This confirms that all $Q-$harmonic functions arise in this way starting from the boundary martingales defined by integration on arcs with respect to the measure $\nu$, and now we know that the same is true for $P-$harmonic functions expressed as the Poisson transforms (induced by $P$) of $\nu-$martingales.

This establishes a one-to-one correspondence between $H_P$ and $H_Q$ whenever $P$ and $Q$ give rise to the same hitting distribution on $\Omega$.  We now extend our results to density properties of martingales, now identified with harmonic functions with respect to a transient transition operator $P$. The key property in the proofs was the uniform decay condition \eqref{eq:decay_when_jumping_to_a_forward_neighbor}, that, by \eqref{eq:factoring_of_U}, is equivalent to an exponential decay of the hitting probability $U$, hence of the Green kernel $G(u,v)$.
 Then our results can be extended to transition operators that allow a similar uniform decay conditions.

\begin{definition}[Very regular transition operators]\label{Def:very_regular_operators}
A nearest-neighbor transition operator $P$ on $T$ is \emph{very regular} if
it satisfies the uniform bounds $p(u,v)\geqslant \delta$ and $p(u,u_-)\leqslant \frac 12 -\delta$ for some $\delta>0$ and all $u\sim v$, $u\neq o$. Note that these bounds force the tree to be locally finite.
\end{definition}

It has been proved in \cite{KPT}
that, for a very regular $P$, a similar uniform bound is satisfied by the probability $U(v_-,v)$ that the random walk induced by $P$ visits $v$ for the first time after starting at $v_-$. Namely,
$U(v_-,v)\leqslant 1-\epsilon$, where $\epsilon>0$ is a universal constant that depends only on the bound $\delta$ of Definition \ref{Def:very_regular_operators}
(namely, $\epsilon=4\delta/(1+2\delta)$). Moreover, for every vertex $v$ the hitting distribution $\nu=\nu_o$ of the random walk starting at $o$ of any transient operator $P$ on the boundary satisfies $\nu(I(v))=U(o,v) (1-U(v,v_-))/(1-U(v_-,v)U(v,v_-))$ and so, for very regular transition operators, the right hand side $k(v):=U(o,v) (1-U(v,v_-))/(1-U(v_-,v)U(v,v_-))$ is of the order of $U(o,v)$ (in the sense that there exist constants $C_+>C_->0$  independent of $v$ such that $C_- < k(v) < C_+$). Note that $k(v) \leqslant U(o,v)$.

It follows easlily  that the results from Section 3 transfer without changes to the context of very regular operators:

\begin{corollary} Theorems
\ref{theo:density_of_Utilde},
\ref{theo:spaceability},
\ref{theo:frequent_density}
and 
\ref{theo:meager_in_martingale_space}
hold, more generally, for each very regular transition operator on $T$.
\end{corollary}

An interesting problem, not considered here, is to which class of transient transition operators that do not satisfy a uniform bound for the probability $U$ (or equivalently for the Green kernel) our results extend. The  argument above applies to the class of transient transition operators for which every vertex $v$ is the starting point of a geodesic ray $[v,v_1, v_2, \dots)$ along which $k(v_j) \neq 0$ and $\prod_j k(v_j) = 0$, or equivalently $\sum_j (1-k(v_j))=\infty$. For this class of operators, all the results of Section \ref{Sec:Menshov}
still hold. A trivial example of a transition operator $P$ whose harmonic functions do not verify our universality results is the operator such that, at each vertex $v_-$, there is only one vertex $v\sim v_-$, $v>v_-$ such that $p(v_-,v)=1$ and all the other transition coefficients are (necessarily) zero. As seen at the beginning of Section \ref{Sec:Menshov},  our universality statements for $T$ equipped with this $P$ are equivalent to the same statements for a sub-tree $T'$ isomorphic to $\mathN$, where they cannot hold.

\end{document}